\documentclass[11pt, reqno]{amsart}
\usepackage[margin=1.3in]{geometry}
\usepackage{amsfonts}
\usepackage{amsthm,amsmath}
\usepackage{amsbsy}
\usepackage{bm}
\usepackage{amssymb} 
\usepackage{verbatim}
\usepackage{times}
\usepackage{xcolor}

\usepackage{units}
\usepackage{enumerate,enumitem}
\usepackage{bbm}
\usepackage{exscale}
\usepackage[hyperindex,breaklinks]{hyperref}
\hypersetup{ colorlinks=true, linkcolor=blue, citecolor=blue,
  filecolor=blue, urlcolor=blue}
  \numberwithin{equation}{section} 
  \date{}

\usepackage[comma, sort&compress]{natbib}

\theoremstyle{plain} 
    \newtheorem{theorem}{Theorem}
    \newtheorem{lemma}[theorem]{Lemma}
    \newtheorem{proposition}[theorem]{Proposition}

\theoremstyle{definition} 

    \newtheorem{remark}[theorem]{Remark}

\DeclareMathOperator{\R}{\mathbb{R}}

\DeclareMathOperator{\Z}{\mathbb{Z}}
\DeclareMathOperator{\N}{\mathbb{N}}

\usepackage{mathtools}

\DeclarePairedDelimiter\floor{\lfloor}{\rfloor}

\DeclareMathOperator{\De}{d}

\newcommand{\e}{\mathrm{e}}

\newcommand{\vr}{\varphi}
\newcommand{\var}{\mathbf{Var}}

\newcommand{\prob}{\mathbf{P}}
\newcommand{\E}{\mathbf{E}}

\newcommand{\la}{\left\langle}
\newcommand{\ra}{\right\rangle}

\definecolor{darkred}{rgb}{1,0,0}
\definecolor{darkgreen}{rgb}{0,0.6,0}
\definecolor{darkblue}{rgb}{0,0,1}
\definecolor{darkagenta}{rgb}{0.6,0,0.6}

\begin{document}
\title{The scaling limit of the $(\nabla+\Delta)$-model}

\author[A. Cipriani]{Alessandra Cipriani}
\address{TU Delft (DIAM), Building 28, van Mourik Broekmanweg 6, 2628 XE, Delft, The Netherlands}
\email{A.Cipriani@tudelft.nl}
\author[B. Dan]{Biltu Dan}
\author[R.~S.~Hazra]{Rajat Subhra Hazra}
\thanks{AC is supported by grant 613.009.102 of the Netherlands Organisation for
Scientific Research (NWO). RSH acknowledges MATRICS grant from SERB and the Dutch stochastics cluster STAR (Stochastics – Theoretical and Applied Research) for an invitation to TU Delft where part of this work was carried out. The authors thank Francesco Caravenna for helpful discussions, and an anonymous referee for insightful remarks and comments on a previous draft of the work.}
\address{ISI Kolkata, 203, B.T. Road, Kolkata, 700108, India}
\email{biltudanmath@gmail.com, rajatmaths@gmail.com}

\date{}

\begin{abstract}
In this article we study the scaling limit of the interface model on $\Z^d$ where the Hamiltonian is given by a mixed gradient and Laplacian interaction. We show that in any dimension the scaling limit is given by the Gaussian free field. We discuss the appropriate spaces in which the convergence takes place. While in infinite volume the proof is based on Fourier analytic methods, in finite volume we rely on some discrete PDE techniques involving finite-difference approximation of elliptic boundary value problems. 
\end{abstract}
\maketitle
\section{Introduction}
The $(\nabla+\Delta)$-model is a special instance of a more general class of random interfaces in which the interaction is governed by the exponential of an energy function $H$, called Hamiltonian. More specifically, random interfaces are fields $\vr=(\vr_{x})_{x\in\Z^{d}}$, whose distribution
is determined by a probability measure on $\mathbb{R}^{\mathbb{Z}^{d}}$,
$d\ge1$. The probability measure is given (formally) by
\begin{equation}\label{eq:Gibbs}
\prob_{\Lambda}(\mathrm{d}\vr):=\frac{\mathrm{e}^{-H(\vr)}}{Z_\Lambda}\prod_{x\in{\Lambda}}\mathrm{d}\vr_{x}\prod_{x\in\mathbb{Z}^{d}\setminus \Lambda}\delta_{0}(\mathrm{d}\vr_{x}),
\end{equation}
where $\Lambda\Subset\mathbb{Z}^{d}$ is a finite subset, $\mathrm{d}\vr_{x}$ is the Lebesgue measure on $\R$, $\delta_{0}$ is the Dirac measure at $0,$ and $Z_{\Lambda}$ is a normalizing constant.
We are imposing zero boundary conditions: almost surely $\vr_{x}=0$
for all $x\in\mathbb{Z}^{d}\setminus{\Lambda}$, but the definition
holds for more general boundary conditions. In this article we consider the special case when the Hamiltonian is given by 
\begin{equation}\label{eq:Ham_def}H(\vr)=\sum_{x\in\Z^d}\left(\kappa_1\|\nabla\vr_x\|^2+\kappa_2(\Delta\vr_x)^2\right)\end{equation}
where $\|\cdot\|$ denotes the Euclidean norm, $\nabla$ is the discrete gradient and $\Delta$ is the discrete Laplacian defined respectively by 
\begin{align*}
 \nabla f(x)&=(f(x+e_i)-f(x))_{i=1}^d\\
 \Delta f(x)&= \frac1{2d}\sum_{i=1}^{d} (f(x+e_i)+f(x-e_i)-2f(x)).
\end{align*}
for any $x\in \Z^d$, $f:\Z^d\to\R$, and $\kappa_1,\,\kappa_2$ are two positive constants. In the physics literature, the above Hamiltonian is considered to be the energy of a semiflexible membrane (or semiflexible polymer if $d=1$) where the parameters $\kappa_1$ and $\kappa_2$ are the lateral tension and the bending rigidity, respectively. The application of Gibbs measures, in particular the $(\nabla+\Delta)$-model, to the theory of biological membranes can be found in~\cite{Leibler:2006, ruiz2005phase, lipowsky1995generic}. In the works of \cite{borecki2010, CarBor} this model was studied in $d=1$ under the influence of pinning in order to understand the localization behavior of the polymer.

The mixed model interpolates between two well-known random interfaces. Indeed, in the purely gradient case ($\kappa_2=0$) one recovers the measure of the discrete Gaussian free field (DGFF). It has great importance in statistical mechanics, and we refer the reader to the reviews by~\cite{Sheff, ASS, ZeiBRW} for further details and existing results. The case of the pure Laplacian interaction, that is, when $\kappa_1=0$, is called membrane or bilaplacian model. It differs from the DGFF in that it lacks a random walk representation for the finite volume covariances, and might have negative correlation. Recent developments around the properties of the model concern its extremes
\citep{Cip13, CCHInterfaces} and the entropic repulsion event handled in \cite{Sakagawa,Kurt_d4}.

In \citet[Remark 9]{CarBor} it was conjectured that, in the case of pinning for the one-dimensional $(\nabla+\Delta)$-model, the behaviour of the free energy should resemble the purely gradient case. In view of this remark it is natural to ask if the scaling limit of the mixed model is dominated by the gradient interaction, that is, the limit is a continuum Gaussian free field (GFF). The main focus of this article is to show that such a guess is true and indeed in any dimension the mixed model approximates the Gaussian free field. We also show that when higher powers of the Laplacian are present in the Hamiltonian, then the model approximates the Gaussian free field.

We will consider the lattice approximation of both domains and $\R^d$ and investigate the behavior of the rescaled interface when the lattice size decreases to zero. We will use techniques coming from discrete PDEs which were already employed in \cite{mm_scaling} to derive the scaling limit of the membrane model. We show that in $d=1$ convergence occurs in the space of continuous functions whilst in higher dimensions the limit is no longer a function, but a random distribution, and convergence takes place in a Sobolev space of negative index. In this sense one can also think of the mixed model as a perturbation of the DGFF. This gives rise to some natural questions which we will state after presenting our main results.

\section{Main results}\label{sec:results}
\subsection{The $(\nabla + \Delta)$-model} Let $\Lambda$ be a finite subset of $\Z^d$ and $\prob_{\Lambda}$ and $H(\vr)$ be as in~\eqref{eq:Gibbs} and \eqref{eq:Ham_def} respectively.  It follows from Lemma 1.2.2 of \cite{Kurt_thesis} that the Gibbs measure~\eqref{eq:Gibbs} on $\R^{\Lambda}$ with Hamiltonian~\eqref{eq:Ham_def} exists. Note that~\eqref{eq:Ham_def} can be written as 
 \begin{equation}\label{eq:ham:mixed}
H(\vr)=\frac{1}{2}\langle\vr,(-4d\kappa_1\Delta+2\kappa_2\Delta^{2})\vr\rangle_{\ell^{2}(\mathbb{Z}^{d})}.
\end{equation}

We are interested in the ``truly'' mixed case, that is when $\kappa_1$ and $\kappa_2$ are strictly positive. 
 For our convenience we will work with the following Hamiltonian:
\begin{equation}
H(\vr)=\frac{1}{2}\langle\vr,(-\kappa_1\Delta+\kappa_2\Delta^{2})\vr\rangle_{\ell^{2}(\mathbb{Z}^{d})}\label{eq:ham}
\end{equation}
where $\kappa_1, \kappa_2$ are positive constants. Thus if we write $G_\Lambda(x,\,y):=\E_\Lambda(\vr_x\vr_y)$, it follows from Lemma 1.2.2 of \cite{Kurt_thesis} that $G_\Lambda$ solves the following discrete boundary value problem:  for $x\in \Lambda$

\begin{equation}\label{eq:cov}
\left\{\begin{array}{lr}
(-\kappa_1\Delta+\kappa_2\Delta^2) G_{\Lambda}(x,y) = \delta_x(y)& y\in \Lambda\\
G_{\Lambda}(x,y) = 0  & y\notin \Lambda\end{array}.\right.
\end{equation}

In the case $\Lambda=[-N,\,N]^d\cap\Z^d$ we will denote the measure~\eqref{eq:Gibbs} by $\prob_N$. It follows from \citet[Proposition 1.2.3]{Kurt_thesis} that in $d\ge 3$ there exists a thermodynamic limit $\prob$ of the measures $\prob_N$ as $N\uparrow\infty$. Under $\prob$, the field $(\vr_x)_{x\in \Z^d}$ is a centered Gaussian process with covariance given by
$$G(x,y)= (-\kappa_1\Delta+\kappa_2\Delta^2)^{-1}(x,\, y).$$
It follows from \citet[Lemma 5.1]{Sakagawa} that $G(x,y)\asymp \|x-y\|^{2-d}$ as $\|x-y\|\to\infty$.


\subsection{Generalizations}
In this subsection we describe some generalisations to the $(\nabla+\Delta)$ model. As discussed before, the $(\nabla+\Delta)$-model forms a basis for modelling biological membranes. In mathematical terms it can be thought of as a special case of a much more general class of models that involve higher powers of the Laplacian in the Hamiltonian. These generalizations were first considered in~\cite{Sakagawa}. Consider the measure in \eqref{eq:Gibbs} with the following formal Hamiltonian:
\begin{equation}\label{gen:ham}
H(\vr)= \frac12\sum_{i=1}^K \frac{\kappa_i}{(2d)^{a_i}} \sum_{x\in \Z^d} \left( (-\Delta)^{i/2}\vr_x\right)^2
\end{equation}
where $K\in \mathbb N$, $\kappa_i\in \mathbb R$, $i=1,\,2,\,\ldots,\, K$ and even, and if $i$ is odd, 
$$ \sum_{x\in \Z^d} \left( (-\Delta)^{i/2}\vr_x\right)^2= \sum_{x}\sum_{j=1}^d\left( \left(-\Delta\right)^{\frac{i-1}{2}}\nabla_j \vr_x\right)^2.$$
Here $\nabla_j\vr_x= \vr_{x+e_j}-\vr_x$. Also $a_i=1$ if $i$ is odd and $0$ if $i$ is even. In general, the behavior of this model depends on $\ell=\min\{i\in \mathbb N: \kappa_i\neq 0\}$. Let 
\[
J=\sum_{i=\ell}^K \kappa_i(-\Delta)^i.
\]

Under the assumption $\sum_{i=\ell}^K \kappa_i r^i>0$ for all $0<r<2$, it follows from \citet[Proposition~1.2.3]{Kurt_thesis} that the Gibbs measure $\prob_{\Lambda}$ on $\mathbb R^{\Lambda}$ with zero boundary conditions outside $\Lambda$ and Hamiltonian~\eqref{gen:ham} exists. The covariance function of the field $G_{\Lambda}(x,y)$ is uniquely defined as the Green's function of a discrete boundary value problem, namely for
$x\in \Lambda$
\begin{equation}\label{eq:cov:gen}
\left\{\begin{array}{lr}
J G_{\Lambda}(x,y) = \delta_x(y)& y\in \Lambda\\
G_{\Lambda}(x,y) = 0  & y\in \partial_K\Lambda \end{array}\right.
\end{equation}
where $\partial_K \Lambda= \{x\in \Lambda^c:\mathrm{dist}(x,\Lambda)\le K\}$ with $\mathrm{dist}(\cdot\,,\cdot)$ being the graph distance in the lattice $\Z^d$.  For this model the thermodynamic limit exists in $d\ge 2\ell+1$. Note that the $(\nabla+\Delta)$-model is a special case when we set $\kappa_1>0$, $\kappa_2>0$ and $\kappa_i=0$ otherwise. We want to study the scaling limit of the model when $\kappa_1=1$ and $\kappa_i\ge 0$ for $i=2,\,\ldots,\, K-1$ and $\kappa_K>0$.

\subsection{Main results} Since the infinite volume measure of the mixed model exists in $d\ge 3$, we split the scaling limit convergence into two parts: the infinite volume case, in which we study the $(\nabla+\Delta)$-model under $\prob$, and the finite volume case in which our object of interest is the scaling limit of measures $\prob_{\Lambda_N}$, for some chosen $\Lambda_N\Subset \Z^d$.  Therefore in the infinite volume case we are going to work with the $(\nabla+\Delta)$-model only for simplicity, whereas in the finite volume case our proof will comprehend all models with Hamiltonian~\eqref{gen:ham}.
We fix once and for all the constant $k:=1/\sqrt{2d}$. The main results are as follows.

\subsection{Infinite volume} In $d\ge 3$ (Section~\ref{sec:big_d}) we consider the infinite volume $(\nabla+\Delta)$-model $\vr=(\vr_{x})_{x\in\Z^{d}}$ with law $\prob$.
For $f\in C_c^\infty(\R^d)$ we define 
\begin{equation}\label{def:fieldS}
\left( \Psi_N, f\right): = \sum_{x\in \frac1N \Z^d} k N^{-\frac{d+2}{2}} \vr_{Nx} f(x).
\end{equation}
We will prove convergence in $\mathcal C^\alpha_{loc}$, the (separable) local Besov--H\"older space with exponent of regularity $\alpha<0$. Roughly speaking, a distribution $\Psi$ is $\alpha$-H\"older regular if for every $x\in\R^d$ and every smooth compactly supported test function $f$ one has
\[
 \lambda^{-d}\Big( \Psi,\,f(\lambda^{-1}(\cdot-x)) \Big)\le C\lambda^\alpha,\quad \lambda\to 0.
\]
Let the operator $(-\Delta)^{-1/2} : C_c^\infty(\R^d)\to L^2(\R^d)$ be defined by
$$(-\Delta)^{-1/2}f(x):= \frac1{(2\pi)^{d/2}}\int_{\R^d} \e^{\iota \la x, \theta\ra } \|\theta\|^{-1} \widehat f(\theta) \De\theta.$$
In this article we choose the following normalization of the Fourier transform:
\[
 \widehat f(\theta):=(2\pi)^{-d/2}\int_{\R^d} \e^{-\iota\la \theta,\,x\ra}f(x) \De x
\]
Observe that
\begin{equation}\label{eq:perseval}\|(-\Delta)^{-1/2}f\|^2_{L^2(\R^d)}=\int_{\R^d} \|\theta\|^{-2} |\widehat f(\theta)|^2 \De \theta.
\end{equation}
We prove the following Theorem.
\begin{theorem}[Scaling limit in $d\ge 3$]\label{thm:big_d}
One has that $\Psi_N\overset{d}\to \Psi$ in the topology of $\mathcal C^\alpha_{loc}$ for every $\alpha<-d$, where $\Psi$ is the Gaussian random field such that for every smooth, compactly
supported function $f$, $(\Psi,\,f)$ is a centered Gaussian with variance
$\|(-\Delta)^{-1/2}f\|_{L^2(\R^d)}^2$.
\end{theorem}

To prove this result, we will first show that for every test function $f$, $(\Psi_N,\,f)$ has Gaussian fluctuations, and then show the tightness of the law of $\Psi_N$ in $\mathcal C^\alpha_{loc}$ by relying on a criterion developed in~\cite{MourratNolen}.
\subsection{Finite volume}\label{subsec:finite} In the finite volume case in $d\ge 2$ (Section~\ref{sec:small_d}) we take $D$ to be a bounded domain (open, connected set) in $\R^d$ with smooth boundary. Let $\Lambda_N\subset \Z^d$ be the largest set satisfying $\Lambda_N\cup \partial_K \Lambda_N\subset N\overline D \cap \Z^d$.  
On $\Lambda_N$ we define the mixed model $\vr$ with law~\eqref{eq:Gibbs} and Hamiltonian \eqref{gen:ham} with $\kappa_1=1$ and $\kappa_i\ge 0$ for $i=2,\,\ldots, \,K-1$ and $\kappa_K>0$. Define $\Psi_N$ by 
\begin{align*}
\Psi_N:=k\sum_{x\in \frac1N\Lambda_N }N^{-\frac{d+2}{2}}\vr_{Nx}\delta_x.
\end{align*}
 One can show $\Psi_N$ is a distribution living in the negative Sobolev space $\mathcal H^{-s}(D)$ for all $s > d$.
To describe the limiting field, there are many equivalent ways to define the Gaussian free field $\Psi_D$ on a domain. One of them is to think of it as a collection of centered Gaussian variables $(\Psi_D, f)$ indexed by $C_c^\infty(D)$ with covariance structure given by 
$$\E[ (\Psi_D, f)(\Psi_D, g)]= \iint_{D\times D} f(x) g(y) G_D(x,y)\De x\De y,\quad f,\,g\in C_c^\infty(D)$$
where $G_D$ is the Green's function of the continuum Dirichlet problem with zero boundary conditions. We now state the main result for the finite volume.

\begin{theorem}[Scaling limit in $d\ge 2$ under finite volume]\label{thm:critical_d}
 $\Psi_N$ converges in distribution to the Gaussian free field $\Psi_D$ as $N\to \infty$ in the topology of $\mathcal H^{-s}(D)$ for $s>d$.
\end{theorem}


A special case for finite volume measures is $d=1$ (Subsection~\ref{sec:d=1}). In this example, the GFF becomes a Brownian bridge, and the type of convergence we obtain is different from all other dimensions (convergence occurs in the space of continuous functions). In this case we consider the model on the ``blow up'' $\Lambda=\Lambda_N$ of an appropriate discretisation of $[0,\,1]$. We define a continuous interpolation $\psi_N$ of the rescaled interface and obtain the following theorem:
\begin{theorem}[Scaling limit in $d=1$]\label{thm:small_d}
 $\psi_N$ converges in distribution to the Brownian bridge on $[0,1]$ in the space $C[0,1]$.
\end{theorem}
As a by-product of this Theorem one can easily obtain the convergence of the discrete maximum in $d=1$.

\subsection{Idea of the proofs} We begin by explaining the idea behind the proof in the infinite volume case (Section~\ref{sec:big_d}).
Recalling~\eqref{eq:perseval}, given the appearance of the Fourier transforms in the limit we write the discrete Green's function in terms of the inverse Fourier transform. We see that a scaling factor appears in such a way the contribution from the $\Delta^2$ factor in the Hamiltonian vanishes, ensuring convergence to a purely gradient model. Tightness will be proven using the tightness criterion proven in~\citet{MourratNolen}. To conclude, a polarisation argument shows that the limit is uniquely identified by the finite-dimensional convergence result.

In the finite volume case we discuss here the idea of the proof for the $(\nabla+\Delta)$-model.
 We show first finite dimensional convergence and secondly tightness. Since the measures are Gaussian the finite dimensional convergence follows from the convergence of the covariance function. However, the behaviour of the covariance of the mixed model is not known explicitly in finite volume (for example, it lacks the classical random walk representation of Ginzburg-Landau models). So we use the expedient of PDE techniques in proving the convergence. The key fact which is used is that the Green's function satisfies the Dirichlet problem~\eqref{eq:cov}. We show that the discrete solution is equal to that of the continuum Dirichlet problem with a negligible error. This approximation is obtained from the interesting approach of~\cite{thomee}. His idea, adapted to our setting, is the following: if we write the operator $(-\Delta+\Delta^2)$ in the rescaled lattice $h\Z^d$ for $h$ small, then due to the scaling we end up dealing with $(-\Delta_h+h^2/(2d) \Delta^2_h)$. To quantify how negligible the presence of $\Delta^2_h$ is, we use some discrete Sobolev inequalities. While dealing with tightness we use a spectral gap argument. We use the fact that the smallest eigenvalue of the negative Laplacian is positive and one can approximate it by the scaled smallest eigenvalue of the discrete operator approximating the negative Laplacian. In Section~\ref{sec:Thomee} we therefore derive these precise estimates, in particular showing how derivatives of the test function appear in the constants. This Section is of independent interest, as it concerns the approximation of PDEs. We remark that our methodology seems to be robust enough to deal with different interface models whenever the interaction is given in terms of a discrete elliptic operator. 

\subsection{Outlook and open problems}
The mixed model gives rise to many interesting mathematical questions. Here we list down a few directions of research on this model.

\begin{enumerate}[leftmargin=*]
\item In \cite{borecki2010, CarBor} the Hamiltonian the authors considered was 
 $$H(\vr)=\sum_{x\in\Z^d}\left(\kappa_1 V_1(\nabla\vr_x)+\kappa_2 V_2(\Delta\vr_x)\right)$$
 where $V_1$ and $V_2$ were potentials with minimal assumptions. In general, it would be interesting to see if the scaling limit of such models under general convexity assumptions behaves in a similar manner to the Ginzburg-Landau models, in particular, if they still converge to the GFF.
\item If one considers the pinned versions of the purely gradient and purely Laplacian model, it is known in different settings that the field exhibits exponential decay of correlations \citep{BoltBrydges,IofVel,BCK}. Can one say the same for the mixed model? 
\item The extremes of the discrete Gaussian free field in $d=2$ are by now well-understood. It is known that the point process of extremes converges to a Cox-cluster process (an overview of the results on this topic is given in \citet{biskup2017extrema}). In $d\ge 3$ on the other hand extremal points behave similarly to the case of independent Gaussian variables (\cite{CCH2015}). We believe that a similar behaviour appears in the mixed model and we will address this issue in a future work.
\item It is known \citep{SS09} that SLE$_4$ arise as scaling limit of the level lines of the DGFF. That is, if one considers the continuous extension of the DGFF with appropriate boundary conditions on a grid approximation of a domain in the complex plane, then the zero-level line converges in distribution, as the grid size goes to 0, to SLE$_4$.  Given our results on the scaling limit in $d=2$ one may ask whether this convergence also holds true in the mixed model setting. 


\end{enumerate}
\paragraph{{\em Structure of the article}}We begin by showing Theorem~\ref{thm:big_d} in Section~\ref{sec:big_d}. The proof of Theorem~\ref{thm:critical_d} is given in Section~\ref{sec:small_d}. We include the one-dimensional Theorem~\ref{thm:small_d} in the section concerning finite volume measures, showing it in Subsection~\ref{sec:d=1}. The estimates on the discrete solution to the Laplacian problem are derived in Section~\ref{sec:Thomee}.

\paragraph{{\em Notation}}In the rest of the paper, $C$ denotes a generic constant that may change from line to line within the same equation.

\section{Infinite volume case}\label{sec:big_d}
In this section we prove Theorem \ref{thm:big_d}. We begin by giving the theoretical setup behind it and then pass to the actual proof.
\subsection{Setup}\label{subsec:setup_big_d}
For the reader's convenience we recall the basics on Besov-H\"older spaces and refer the reader to~\cite{FurlanMourrat} for more specific details. Let $U\subset \R^d$, $r\in \mathbb N\cup \{\infty\}$ and $C^r(U)$ denote the set of $r$ times continuously differentiable functions on $U$, and $C^r_c(U)$ the functions of the above space with compact support. For $f\in C^r(\R^d)$, denote
$$\|f\|_{C^r}:= \sum_{|\alpha|\le r} \|D^\alpha f\|_{L^\infty}.$$
The Besov-H\"older spaces are defined as follows. Let $\alpha<0$ and $r_0=-\lfloor \alpha \rfloor$ and
$$\mathcal B^{r_0}=\{\eta\in C^{r_0}: \|\eta\|_{C^{r_0}}\le 1 \text{ and } \mathrm{Supp}(\eta)\subset B(0,1)\}.$$ For $f\in C_c^\infty(\R^d)$ denote
$$\|f\|_{\mathcal C^\alpha}:=\sup_{\lambda\in (0,1]}\sup_{x\in \R^d}\sup_{\eta\in \mathcal B^{r_0}}\lambda^{-\alpha} \int_{\R^d}\lambda^{-d} f(y) \eta\left(\frac{y-x}{\lambda}\right)\De y.$$
The Besov--H\"older space $\mathcal C^\alpha_{loc}$, which henceforth we will abbreviate as $\mathcal C^\alpha$, is the completion of $C_c^\infty(\R^d)$ with respect to the norm $\|\cdot\|_{\mathcal C^\alpha}$. 
Let $\left(\vr_{x}\right)_{x\in\Z^{d}}$ be the infinite volume mixed
model in $d\ge 3.$ Let $f\in C_{c}^{\infty}(\R^{d})$ be a smooth
and compactly supported test function. Construct the field $\Psi_{N}$
as acting on test functions as follows:
\[
\left(\Psi_{N},f\right):=k N^{-\frac{d}{2}-1}\sum_{x\in\Z^{d}}f(N^{-1} x)\vr_x.
\]
In the following we will divide the proof of Theorem~\ref{thm:big_d} in two parts. In Subsection~\ref{subsec:proof_big_d} we will show the convergence of the marginal law of $\left(\Psi_{N},f\right)$ for any $f\in C_c^\infty(\R^d)$, in Subsection~\ref{subsec:Holderspace} we will show tightness. A polarization argument allows to deduce from this the limiting field as in~\citet[Remark~1.4]{MourratNolen}.

\subsection{Fluctuations of $(\Psi_N,\,f)$}\label{subsec:proof_big_d}
In this subsection, we will prove that as $N\to\infty$ for any smooth and compactly supported function $f$ one has
\begin{equation}\label{eq:conv_dist_char}
 (\Psi_N,\,f)\stackrel{d}{\to}\mathcal N\Big(0,\,\|(-\Delta)^{-1/2}f\|^2_{L^2(\R^d)}\Big).
\end{equation}
Given the Gaussian nature of the variables we consider, and the fact that they are centered, it suffices to show that for any such $f$
	\begin{equation}\E\left[\left( \Psi_N, f\right)^2\right]\rightarrow \|(-\Delta)^{-1/2}f\|^2_{L^2(\R^d)}.\label{eq:var_d3}\end{equation}
	We will begin with a preliminary lemma.
	\begin{lemma}\label{mu_bound}
  There exists a constant $C>0$ such that for all $\theta \in [-N\pi, N\pi]^d\setminus \{0\}$
  \begin{align*}
  N^{-2}\left( \frac{\|\theta\|^2}{2dN^2} + \frac{\|\theta\|^4}{4d^2N^4}\right)^{-1}\le N^{-2}\left(\mu\Big(\frac\theta{N}\Big)+\mu\Big(\frac\theta{N}\Big)^2\right)^{-1} \le \frac{2d}{\|\theta\|^2} + \frac{Cd}{2 N^2}.
  \end{align*}
  \end{lemma}
  \begin{proof}\label{pg:proof_mu_bound}
  We know from \citet[Lemma~7]{CHR} that
	there exists $C>0$ such that for all $N\in \N$ and $w\in [-N\pi/2, N\pi/2]^d\setminus \{0\}$
	\begin{align}
	\frac{1}{\| w\|^4}\le N^{-4}\left(\sum_{i=1}^d \sin^2\left(\frac{w_i}{N}\right)\right)^{-2}\le \left(\frac{1}{\|w\|^2}+\frac{C}{N^{2}}\right)^2.
	\end{align}
	Therefore $$\left(\frac{2dN^2}{\|\theta\|^2} + \frac{Cd}2\right)^{-1}\le \mu\Big(\frac\theta{N}\Big)\le \frac{\|\theta\|^2}{2dN^2} $$
	and hence
	\begin{align*}
	N^{-2}\left( \frac{\|\theta\|^2}{2dN^2} + \frac{\|\theta\|^4}{4d^2N^4}\right)^{-1}&\le N^{-2}\left(\mu\Big(\frac\theta{N}\Big)+\mu\Big(\frac\theta{N}\Big)^2\right)^{-1}\\
	 & \le N^{-2}\left(\left(\frac{2dN^2}{\|\theta\|^2} + \frac{Cd}2\right)^{-1}+\left(\frac{2dN^2}{\|\theta\|^2} + \frac{Cd}2\right)^{-2}\right)^{-1}\\
	& \le \frac{2d}{\|\theta\|^2} + \frac{Cd}{2 N^2}.\qedhere
	\end{align*}
  \end{proof}
  We can now begin with the proof of~\eqref{eq:var_d3}.
By definition of the field and translation invariance we have that
	\begin{align}
	\E\left[ \left( \Psi_N,  f\right)^2 \right]&=k^2N^{-(d+2)} \sum_{x,y\in \frac1N\Z^d} \E[ \vr_{Nx}\vr_{Ny}] f(x) f(y) \nonumber\\
    &=k^2N^{-(d+2)} \sum_{x,y\in \frac1N \Z^d} G(0, N(y-x))f(x) f(y).\label{eq:green1}
   \end{align}
Now our goal is to shift these expression to Fourier coordinates. We deduce from the Fourier inversion formula, in the same fashion of \citet[Lemmas~1.2.2,~1.2.3]{Kurt_thesis}, that
\begin{equation}\label{eq:green:inversion}
G(0,x)= \frac1{(2\pi)^d} \int_{[-\pi, \pi]^d} \left(\mu(\theta)+\mu(\theta)^2\right)^{-1}\e^{-\iota\la x, \theta\ra} \De \theta
\end{equation}	
where $\mu(\theta)=\frac1{d}\sum_{i=1}^d (1-\cos(\theta_i))$. 
 Returning to the expression~\eqref{eq:green1} and plugging in~\eqref{eq:green:inversion} we have
	\begin{align}
	\E&\left[ \left( \Psi_N, f\right)^2 \right]\nonumber\\
	&=\frac{ k^2N^{-(d+2)}}{(2\pi)^d} \sum_{x,y\in  \frac1N \Z^d}\int_{[-\pi, \pi]^d} \left(\mu(\theta)+\mu(\theta)^2\right)^{-1} \e^{-\iota\la N(y-x), \theta\ra} f(x) f(y) \De\theta\nonumber\\
	&=\frac{ k^2N^{-2}}{(2\pi)^d}\int_{[-N\pi, N\pi]^d}  \left(\mu\left(\frac\theta{N}\right)+\mu\left(\frac\theta{N}\right)^2\right)^{-1} \left|N^{-d} \sum_{x\in  \frac1N \Z^d} \e^{-\iota\la x, \theta \ra} f(x)\right|^2\De \theta.\nonumber\\
	\end{align}
Here we exchange sum and integral due to Lemma~\ref{mu_bound}. Then we notice that
  \begin{align}
	\lim_{N\to+\infty} \int_{[-N\pi, N\pi]^d} &\Bigg[ N^{-2} \Bigg(\mu\bigg(\frac\theta{N}\bigg)+\mu\bigg(\frac\theta{N}\bigg)^2\Bigg)^{-1} -2d\|\theta\|^{-2}\Bigg]\times\nonumber\\
	\times&\Bigg|(2\pi)^{-d/2}N^{-d} \sum_{x\in  \frac1N \Z^d} \e^{-\iota\la x, \theta \ra} f(x)\Bigg|^2\De \theta=0.\label{claim:error1}
 \end{align}
 In fact, by Lemma \ref{mu_bound} we can sandwich the expression in~\eqref{claim:error1} between two infinitesimal quantities. The lower bound is given by
\begin{equation}
\int_{[-N\pi, N\pi]^d} \left[ N^{-2}\left( \frac{\|\theta\|^2}{2dN^2} + \frac{\|\theta\|^4}{4d^2N^4}\right)^{-1} -2d\|\theta\|^{-2}\right] \left|(2\pi)^{-d/2}N^{-d} \sum_{x\in  \frac1N \Z^d} \e^{-\iota\la x, \theta \ra} f(x)\right|^2\De \theta \label{eq:lower}
\end{equation}
and the upper bound is given by
\begin{equation}
\int_{[-N\pi, N\pi]^d} \frac{Cd}{2 N^2} \left|(2\pi)^{-d/2}N^{-d} \sum_{x\in  \frac1N \Z^d} \e^{-\iota\la x, \theta \ra} f(x)\right|^2\De \theta. \label{eq:upper}
\end{equation}

We show that both the limit of~\eqref{eq:lower} and~\eqref{eq:upper} are zero as $N\to\infty$. Using Lemma 4.7 of \cite{mm_scaling} we have that for any $N$ and $s>d$ 
\begin{equation}\label{eq:decayRS}
\left|(2\pi)^{-d/2} N^{-d}\sum_{x\in \Z^d} \e^{-\iota \la \frac{x}{N}, \theta\ra} f\left(\frac{x}{N}\right)-\widehat f(\theta)\right|\le CN^{-s}.
\end{equation}

Using~\eqref{eq:decayRS} it follows that~\eqref{eq:upper} converges to zero. For~\eqref{eq:lower} observe that the integrand goes to zero and we can apply the dominated convergence theorem due to the following integrable bound:
\begin{align*}
&\Bigg|\left[ N^{-2}\left( \frac{\|\theta\|^2}{2dN^2} + \frac{\|\theta\|^4}{4d^2N^4}\right)^{-1} -2d\|\theta\|^{-2}\right] \Bigg|\Bigg|(2\pi)^{-d/2}N^{-d} \sum_{x\in  \frac1N \Z^d} \e^{-\iota\la x, \theta \ra} f(x)\Bigg|^2\\
&\le \Bigg|\Bigg[ N^{-2}\left( \frac{\|\theta\|^2}{2dN^2} + \frac{\|\theta\|^4}{4d^2N^4}\right)^{-1} -2d\|\theta\|^{-2}\Bigg] 2\left(CN^{-2s}+|\widehat f(\theta)|^2\right)\Bigg|\\
&\le \frac{8d}{\|\theta\|^2}\left(CN^{-2s}+|\widehat f(\theta)|^2\right).\end{align*}
This shows~\eqref{claim:error1}.
Next the following convergence holds from the estimates~\eqref{eq:decayRS} and~\eqref{eq:perseval}:

	\begin{equation}\label{claim:RS}\lim_{N\to+\infty}\frac1{(2\pi)^{d}}\int_{[-N\pi, N\pi]^d}\|\theta\|^{-2}\Bigg|N^{-d} \sum_{x\in  \frac1N \Z^d} \e^{-\iota\la x, \theta \ra} f(x)\Bigg|^2\De \theta=\|(-\Delta)^{-1/2}f\|^2_{L^2(\R^d)}.
	\end{equation}
	Thus for all $f\in C_c^\infty(\R^d)$ the convergence in distribution ~\eqref{eq:conv_dist_char} follows. 
\subsection{Tightness in Besov--H\"older spaces}\label{subsec:Holderspace}We now state the criterion we will employ to show tightness in Besov--H\"older spaces.
\begin{proposition}[{\citet[Proposition~3.1]{MourratNolen}}]
Let $f\in C_{c}^\infty(\R^{d})$ and  $f_{\lambda}(x):=\lambda^{-d}f(\lambda/x).$ Let $(\Psi_N)_{N\in\N}$ be a sequence of stationary random distributions.
Assume that for all $p\ge1$ there exists $C=C(p,\,f)$ such that for all
$N,\,\lambda\in(0,\,1]$ one has
\[
E\left[\left|\left(\Psi_{N},f_{\lambda}\right)\right|^{p}\right]^{\frac{1}{p}}\le C\lambda^{-d}.
\]
Then $\Psi_{N}$ is tight in $\mathcal C^\alpha,$ $\alpha<-d$.
\end{proposition}
Given the previous Proposition, we can now begin to show tightness.

We start with the variance, that is, with $p=2$. We have
\begin{align}
E\left[\left|\left(\Psi_{N},f_{\lambda}\right)\right|^{2}\right] & =\kappa^2 N^{-(d+2)}\sum_{x,\,y\in\Z^{d}}f_{\lambda}(N^{-1} x)f_{\lambda}(N^{-1} y)E\left[\vr_x\vr_y\right]\nonumber \\
 & =\kappa^2 N^{-(d+2)}\sum_{x,\,y\in\Z^{d}}f_{\lambda}(N^{-1} x)f_{\lambda}(N^{-1} y)G(0,\,x-y)\label{eq:var}
\end{align}
by translation invariance of the field. Note that in swapping expectation
and limit we are using the fact that $f(\cdot)$ has compact support.
Plugging \eqref{eq:green:inversion} in \eqref{eq:var} we see that the right-hand
side of \eqref{eq:var} equals, after a change of variables,
\begin{align}
&\frac{N^{-2}}{(2\pi)^{d}}  \int_{\left[-{N\pi},\,N\pi\right]^{d}}N^{-2d}\sum_{x,\,y\in N^{-1}\Z^{d}}f_{\lambda}(x)f_{\lambda}(y)\frac{\e^{-i\langle x-y,\,\theta\rangle}}{\mu(N^{-1}\theta)+\mu(N^{-1}\theta)^{2}}\De\theta\\
&=\frac{N^{-2}}{(2\pi)^{d}}\int_{\left[-N\pi,\,N\pi\right]^{d}}\frac{1}{\mu(N^{-1}\theta)+\mu(N^{-1}\theta)^{2}}\left|N^{-d}\sum_{x\in  N^{-1}\Z^{d}}f_{\lambda}(x)\e^{-i\langle x,\,\theta\rangle}\right|^{2}\De\theta\nonumber .
\end{align}
By Lemma~\ref{mu_bound} we obtain a further upper bound by 
\begin{equation}
\frac{1}{(2\pi)^{d}}\int_{\left[-N\pi,\,N\pi\right]^{d}}\Big(\frac{2d}{\|\theta\|^{2}}+\frac{Cd}{2N^2}\Big)\left|\left(N\lambda\right)^{-d}\sum_{x\in N^{-1} \Z^{d}}f(x/\lambda)\e^{-i\langle x,\theta\rangle}\right|^{2}\De\theta.\label{eq:upbd}
\end{equation}
 Using the Poisson summation formula as in~\citet[Lemma 4.7]{mm_scaling} (the result there is stated for Schwartz functions but can be quickly extended to smooth compactly supported functions) the following estimate follows. Let $s>d$. Then
\[
\left|(2\pi)^{-d/2}\left(N\lambda\right)^{-d}\sum_{x\in N^{-1} \Z^{d}}f(x/\lambda)\e^{-i\langle x,\,\theta\rangle}-\widehat{f}(\lambda\theta)\right|\le C\left(N\lambda\right)^{-s}
\]

We thus obtain a
further upper bound of \eqref{eq:upbd} as

\begin{align}
&\int_{\left[-N\pi,\,N\pi\right]^{d}}\Big(\frac{2d}{\|\theta\|^{2}}+\frac{Cd}{2N^2}\Big)\left(\left|\widehat{f}(\lambda\theta)\right|+C\left(N\lambda\right)^{-s}\right)^2\De\theta\nonumber\\
&\le \int_{\left[-N\pi,\,N\pi\right]^{d}}\Big(\frac{2d}{\|\theta\|^{2}}+\frac{Cd}{2N^2}\Big)\left(2\left|\widehat{f}(\lambda\theta)\right|^2+C\left(N\lambda\right)^{-2s}\right)\De\theta.\label{eq:red_eq}
\end{align}
 We have from~\eqref{eq:red_eq}  
two summands: the first by the change of variables $\theta':=\theta/\lambda$ being
\begin{align*}
 & \int_{\left[-N\pi,\,N\pi\right]^{d}}\Big(\frac{2d}{\|\theta\|^{2}}+\frac{Cd}{2N^2}\Big)\left|\widehat{f}(\lambda\theta)\right|^2\De\theta\\
 &\le  C\lambda^{2-d}\int_{\R^d}\frac{1}{\|\theta\|^{2}}\left|\widehat{f}(\theta)\right|^2\De\theta + CN^{-2}\lambda^{-d}\int_{\R^d}\left|\widehat{f}(\theta)\right|^2\De\theta\\
 & \le C \big(\lambda^{2-d} + \lambda^{-d}\big) \le C\lambda^{-d},
\end{align*}
since $\left\Vert \theta\right\Vert ^{-2}$ is integrable at $0$
in $d\ge3$ and $\widehat f(\cdot)$ decays faster than any polynomial at infinity;
the second being
\begin{align*}
{C\left(N\lambda\right)^{-2s}} & \int_{\left[-N\pi,\,N\pi\right]^{d}}\Big(\frac{2d}{\|\theta\|^{2}}+\frac{Cd}{2N^2}\Big)\De\theta
\le C N^{-2s} N^{d-2}\lambda^{-2s}\le C \lambda^{-2s}
\end{align*}
for $s>(d-2)/2$.
Therefore we have
\[
E\left[\left|\left(\Psi_{N},f_{\lambda}\right)\right|^{2}\right]\le C\lambda^{-d}\vee \lambda^{-2s}.
\]
Since we need that $s>d$, we have
\[
E\left[\left|\left(\Psi_{N},f_{\lambda}\right)\right|^{2}\right]^{\frac{1}{2}}\le C\lambda^{-d}.
\]
Since $\Psi_{N}$ is Gaussian, one can see that for all $p=2m,\,m\in\N,$
one has
\[
E\left[\left|\left(\Psi_{N},f_{\lambda}\right)\right|^{p}\right]^{\frac{1}{p}}\le CE\left[\left|\left(\Psi_{N},f_{\lambda}\right)\right|^{2}\right]^{\frac{m}{2m}} \le C\lambda^{-d}
\]
from which the result follows by extending the bound to any $p\ge1$
with H\"older's inequality. Hence tightness follows.

\begin{remark}The previous result can be adapted to prove the convergence of $\Psi_N$ to $\Psi$ in the space $\mathcal S^*(\R^d)$, the dual of the space of Schwartz functions. By \citet[Corollary~2.4]{BOY2017} the convergence of the characteristic function is sufficient to determine the limiting field. In the case when $f\in\mathcal S$, the proof of~\eqref{eq:var_d3} can be carried out in exactly the same way as for $C_c^\infty$ test functions, thanks to the rapid decay of the Fourier transform of Schwartz functions and the Poisson summation formula. We omit the details and refer the readers to \cite{mm_scaling} where this set-up was used for infinite volume membrane model.\end{remark}

%

\section{Finite volume case}\label{sec:small_d}
\subsection{Setup}\label{subsec:preliminary}
In this Section we will consider in details the finite volume limit of interfaces with Hamiltonian~\eqref{gen:ham} in the case $\kappa_1=1,\,\kappa_i\ge 0$ for $i=2,\,\ldots,\,K-1$ and $\kappa_K>0$. We will now show the finite dimensional convergence.

 Let $D$ be any bounded domain in $\R^d$ with smooth boundary.  Let $D_N$ and $\Lambda_N$ be as defined in Subsection~\ref{subsec:finite}. The key result of this Subsection is to show that the variance of $(\Psi_N,\,f)$ converges to that of $(\Psi_D,\,f)$, that is, to the norm of the solution of a suitable Dirichlet problem.
\begin{remark}
The reduction from smooth boundary to piece-wise smooth boundaries can perhaps be achieved but we will not aim for such a generalization in this article. 
\end{remark}
\begin{proposition}\label{f1}
Let $f$ be a smooth and compactly supported function on $D$ and consider 
$$(\Psi_N,\,f)=k\sum_{x\in \frac1N\Lambda_N }N^{-\frac{d+2}{2}}\vr_{Nx}f(x).$$ Then 
$$\lim_{N\to\infty}\var[(\Psi_N,\,f)] = \int_D u(x)f(x)\De x ,$$ where $u$ is the solution of the Dirichlet problem
\begin{equation}
\begin{cases}
-\Delta_c u(x) = f(x) & x\in D\\
u(x)=0 & x\in\partial D 
\end{cases}\label{eq2:continuum}
\end{equation}
and $\Delta_c$ is the Laplace operator defined by $\Delta_c=\sum_{i=1}^d\frac{\partial^2}{\partial x^2_i}$.
\end{proposition}

\begin{proof}
We denote $G_{\frac1N}(x,y):=\E_{\Lambda_N}[ \vr_{Nx}\vr_{Ny}]$ for $x,y\in N^{-1} D_N$. Note that if $\Delta_{\frac1N}$ is the discrete Laplacian on $N^{-1}\Z^d$ then by~\eqref{eq:cov:gen} we have, for all $x\in N^{-1}\Lambda_N$,
\begin{equation}
\begin{cases}
\Big(\sum_{i=1}^K\frac{\kappa_i}{(2dN^2)^i}(-\Delta_{\frac1N})^i\Big)  G_{\frac1N}(x,y) = \delta_{x}(y)& y\in \frac1N\Lambda_N \\
G_{\frac1N}(x,y) = 0 & y\notin \frac1N\Lambda_N.
\end{cases}
\label{eq2:G_N}
\end{equation}
We have 
\begin{align*}
\var[(\Psi_N,\,f)]&= k^2\sum_{x, y\in \frac1N\Lambda_N} N^{-d-2} G_{\frac1N}(x,y) f(x) f(y)\\
&=\sum_{x\in \frac1N\Lambda_N} N^{-d} H_N(x) f(x)
\end{align*}
where $H_N(x)= k^2\sum_{y\in \frac1N\Lambda_N}N^{-2} G_{\frac1N}(x,y) f(y)$ for $x\in N^{-1} D_N$. It is immediate from \eqref{eq2:G_N} that $H_N$ is the solution of the following Dirichlet problem:
\begin{equation}\label{eq:dBVP}
 \begin{cases}
\Big(\sum_{i=1}^K\frac{\kappa_i}{(2dN^2)^i}(-\Delta_{\frac1N})^i\Big) H_N(x) = f(x)& x\in \frac1N\Lambda_N\\
H_N(x)= 0 & x\notin \frac1N\Lambda_N.
\end{cases}
\end{equation}
Define the error between the solutions of~\eqref{eq:dBVP} and \eqref{eq2:continuum} by $e_N(x):= H_N(x)-u(x)$ for $x\in N^{-1} D_N$. We use the estimate given in Theorem \ref{thm:one} and get
 \begin{equation}\label{eq:thomee2}
N^{-d} \sum_{x\in \frac1N\Lambda_N} e_N(x)^2\le CN^{-1}.
\end{equation}
Rewriting the variance we deduce
$$\var[(\Psi_N,\,f)]= \sum_{x\in \frac1N\Lambda_N} e_N(x) f(x)N^{-d} + \sum_{x\in \frac1N\Lambda_N} u(x) f(x) N^{-d}.$$
Note that by Cauchy-Schwarz inequality and \eqref{eq:thomee2} the first summand goes to zero as $N\to \infty$. The second term is a Riemann sum and converges to $\int_{D} u(x) f(x)\De x.$\qedhere
\end{proof}

\subsection{The continuum Gaussian free field}\label{subsec:defGFF}
In this case we consider $d\ge 2$ and $D$ and $\Lambda_N$ as in the previous Subsection. First we discuss briefly some definitions about the GFF. In $d=2$ the results can be found already in the literature, see for example \citet[Section 1.3]{bere}.

By the spectral theorem for compact self-adjoint operators we know that there exist eigenfunctions $(u_j)_{j\in\N}$ of $-\Delta_c$ corresponding to the eigenvalues $0<\lambda_1\le\lambda_2\le\ldots \to\infty$ such that $(u_j)_{j\ge 1}$ is an orthonormal basis of $L^2(D)$. By elliptic regularity, we have that $u_j$ is smooth for all $j$. Let $s>0$ and we define the following inner product on $C_c^\infty(D)$:
$$\langle f,\,g \rangle_s:=\sum_{j\in\N} \lambda_j^s\langle f\,,\,u_j\rangle_{L^2}\langle u_j\,,\,g\rangle_{L^2} .$$ 
Then $\mathcal H^s_0(D)$ can be defined to be the completion of $C_c^\infty(D)$ with respect to this inner product and $\mathcal H^{-s}(D)$ is defined to be its dual. Here we note that $\mathcal H^s_0(D) \subset L^2(D)\subset \mathcal H^{-s}(D)$ for any $s>0$. 
 
In case $f\in L^2(D)$ then we have $$\| f \|_{-s}^2=\sum_{j\in \N}\lambda_j^{-s}\la f \,,\,u_j\ra_{L^2}^2.$$

Also observe that $(\lambda_j^{-1/2}u_j)_{j\in\N}$ is an orthonormal basis of $\mathcal H^1_0(D)$. In the following Proposition we give the definition of the zero boundary continuum Gaussian free field $\Psi_D$ via its Wiener series, generalising the two-dimensional result of~\citet[Subsection 4.2]{Dubedat}. 
 
\begin{proposition}\label{prop:series_rep_h}
	Let $(\xi_j)_{j\in \N}$ be a collection of i.i.d. standard Gaussian random variables. Set the GFF with zero boundary conditions outside $D$ to be
	\[
	\Psi_D:=\sum_{j\in \N}\lambda_j^{-1/2}\xi_j u_j.
	\]
	Then $\Psi_D\in \mathcal H^{-s}(D)$ a.s. for all $s>{d}/2 -1$.
\end{proposition}

\begin{proof}
Fix $s>{d}/2 -1$. Clearly $u_j\in L^2(D)\subseteq \mathcal H^{-s}(D)$. We want to show that $\|\Psi_D\|_{-s} < \infty$ with probability one. We have
\begin{align*}
	\|\Psi_D\|^2_{-s} =\sum_{j\in \N}\lambda_j^{-1-s} \xi^2_j.
\end{align*}
The last sum is finite a.s. by Kolmogorov's two series theorem as we have 
\begin{align*}
\sum_{j\in \N}\E[\lambda_j^{-1-s}\xi^2_j] \asymp  \sum_{j\in \N}j^{-\frac2d(1+s)} <\infty
\end{align*}
and 
\begin{align*}
\sum_{j\in \N}\var[\lambda_j^{-1-s}\xi^2_j] \asymp \sum_{j\in \N}j^{-\frac4d(1+s)} <\infty.
\end{align*}
Here we have used the Weyl's asymptotic $\lambda_j\sim Cj^{\frac2d}$ for some explicit constant $C$. Thus we have $\Psi_D\in \mathcal H^{-s}(D)$ a.s.
\end{proof}

\subsection{Proof of Theorem~\ref{thm:critical_d}}\label{subsec:thm2finite}
We are now ready to show the main result on the scaling limit in the finite volume case. All notations are borrowed from Subsections~~\ref{subsec:preliminary}-\ref{subsec:defGFF}.

\begin{proof}[Proof of Theorem \ref{thm:critical_d}]
We first show that for $f\in C_c^\infty(D)$ 
\begin{equation}
\label{eq:toshowfdm}
( \Psi_N, f)\overset{d}\rightarrow (\Psi_D\,,\,f).\end{equation}
This follows from the following two observations: on the one hand by Proposition \ref{f1} and integration by parts we obtain 
$$\var[( \Psi_N, f)] \rightarrow \int_D u(x)f(x)\De x= \|f\|_{-1}^2.$$
On the other hand from the definition of GFF it follows that
 \begin{align*}
\var[(\Psi_D\,,\,f)]=\sum_{j\in \N}\lambda_j^{-1} \la u_j\,,\,f\ra_{L^2}^2=\|f\|_{-1}^2.
\end{align*}
Consequently we obtain~\eqref{eq:toshowfdm} since both $( \Psi_N, \,f)$ and $( \Psi_D, \,f)$ are centered Gaussians.

Next we want to show that the sequence $(\Psi_N)_{N\in \mathbb{N}}$ is tight in $\mathcal H^{-s}(D)$ for all $s > d$. It is enough to show that
\begin{equation}\label{limsup_psi_tight}
\limsup_{N\to \infty}\E_{\Lambda_N}[\|\Psi_N\|_{-s}^2]<\infty \quad \forall \, s>d.
\end{equation}
The tightness of $(\Psi_N)_{N\in \mathbb{N}}$ would then follow immediately from \eqref{limsup_psi_tight} and the fact that, for $0\le s_1<s_2$, $\mathcal H^{-s_1}(D)$ is compactly embedded in $\mathcal H^{-s_2}(D)$. 
In order to show \eqref{limsup_psi_tight} we first observe that for any $f\in\mathcal H^s_0(D)$
\begin{align*}
|( \Psi_N, f)|&=\Big{|}k\sum_{x\in \frac1N\Lambda_N }N^{-\frac{d+2}{2}}\vr_{Nx}\sum_{j\ge 1} \langle f\,,\,u_j\rangle_{L^2} u_j(x)\Big{|}\\
&=kN^{-\frac{d+2}{2}}\Big{|}\sum_{j\ge 1} \lambda_j^{-\frac{s}2}\sum_{x\in \frac1N\Lambda_N }\vr_{Nx}u_j(x) \lambda_j^{\frac{s}2}\langle f\,,\,u_j\rangle_{L^2}\Big{|}\\
&\le kN^{-\frac{d+2}{2}} \left(\sum_{j\ge 1} \lambda_j^{-s}\Bigg(\sum_{x\in \frac1N\Lambda_N }\vr_{Nx}u_j(x)\Bigg)^2\right)^{\frac12} \|f\|_{s}
\end{align*}
where in the first equality we have used the fact that $f\in L^2(D)$ and therefore $f=\sum_{j\ge 1} \langle f\,,\,u_j\rangle_{L^2} u_j$.
Thus we have, using the definition of dual norm,
\begin{align*}
\|\Psi_N\|_{-s}^2 \le \sum_{j\ge 1} \lambda_j^{-s} k^2 N^{-(d+2)}\Bigg(\sum_{x\in \frac1N\Lambda_N }\vr_{Nx}u_j(x)\Bigg)^2.
\end{align*}
By monotone convergence we obtain
\begin{align}
\E_{\Lambda_N}\|\Psi_N\|_{-s}^2 &\le \sum_{j\ge 1} \lambda_j^{-s} k^2 N^{-(d+2)}\sum_{x,y\in \frac1N\Lambda_N }G_{\frac1N}(x,y) u_j(x)u_j(y)\nonumber\\
& \le \sum_{j\ge 1} \lambda_j^{-s} k^2 N^{-2} \|G_{\frac1N}u_j\|_{\ell^2(\frac1N\Lambda_N )}\|u_j\|_{\ell^2(\frac1N\Lambda_N )}\label{eq:tight:inequality}
\end{align}
where for any grid function $f$ we define 
$$\| f\|^2_{\ell^2(\frac1N\Lambda_N )}:= N^{-d} \sum_{x\in \frac1N\Lambda_N } f(x)^2.$$
From~\eqref{eq2:G_N} it follows that $G_{\frac1N}$ is the Green's function for $\Big(\sum_{i=1}^K\frac{\kappa_i}{(2dN^2)^i}(-\Delta_{\frac1N})^i\Big)$. Let $\nu_1, \nu_2,\ldots$ be the eigenvalues of $G_{\frac1N}$. Define $P_i$ to be the projection on the $i$-th eigenspace. Then using orthogonality we have
\begin{equation}\label{eq:green:max}
\|G_{\frac1N}u_j\|^2_{\ell^2(\frac1N\Lambda_N )} = \sum_{i} \nu_i^2\|P_iu_j\|^2_{\ell^2(\frac1N\Lambda_N )}
\le \nu^2_{max} \|u_j\|^2_{\ell^2(\frac1N\Lambda_N )}
\end{equation}
where $\nu_{max}$ is the largest eigenvalue of $G_{\frac1N}$. Using ~\eqref{eq:green:max} in ~\eqref{eq:tight:inequality} we obtain
\begin{align*}
\E_{\Lambda_N}\|\Psi_N\|_{-s}^2 &\le \sum_{j\ge 1} \lambda_j^{-s} k^2 N^{-2} \nu_{max}\|u_j\|_{\ell^2(\frac1N\Lambda_N )}^2\\
&\le C \sum_{j\ge 1} \lambda_j^{-s} k^2 N^{-2} \nu_{max}\left( \sup_{x\in D}u_j(x)\right)^2.
\end{align*}
From Theorem 1.4 in \cite{BV1999} we know that for any $x\in D$, $|u_j(x)|\le \lambda_j^{{d}/{4}}$.  On the other hand from Theorem \ref{thm:spec_gap} we know that $\lambda_1^{-1} $ is approximated by $ N^{-2}\nu_{\max}$, therefore $N^{-2}\nu_{\max}$ is bounded above (as $\lambda_1$ is bounded away from zero). Using these observations we have
\begin{align*}
\limsup_{N\to\infty}\E_{\Lambda_N}\|\psi_N\|_{-s}^2 \le C \sum_{j\ge 1} \lambda_j^{-s+\frac{d}2}.
\end{align*}
The last sum is finite whenever $s>d$.

Thus we have proved \eqref{limsup_psi_tight}.~A standard uniqueness argument using the facts that $\mathcal H^{-s}(D)$ is the topological dual of $\mathcal H^{s}_0(D)$ and $C_c^\infty(D)$ is dense in $\mathcal H_0^s(D)$ (see proof of Theorem 3.11 of \cite{mm_scaling}) completes the proof of Theorem~\ref{thm:critical_d}. 
\end{proof}

\subsection{One-dimensional case}\label{sec:d=1}




In this case for simplicity we consider $D=(0,1)$ and the corresponding $D_N$, $\Lambda_N$ and the model as defined in Subsection~\ref{subsec:preliminary}, in particular $\Lambda_N=\{2,\,\ldots,\,N-2\}$. To study the scaling limit we define a continuous interpolation $\psi_N$ for each $N$ as follows:
\begin{align*}
\psi_N(t)=kN^{-\frac12}\left[\vr_{\floor{Nt}} + (Nt-\floor{Nt})(\vr_{\floor{Nt}+1}-\vr_{\floor{Nt}})\right] , \,\,\,\,t\in \overline D.
\end{align*}

The proof of Theorem \ref{thm:small_d} follows if one can show tightness in $C[0,1]$ and the finite dimensional convergence. The tightness follows from bounds on the variance of the increments of the process by that of the increments of the Gaussian free field. Then one can apply Theorem 14.9 of \cite{kallenberg:foundations}.  
The bound on the variance of the increments is obtained by an application of the Brascamp--Lieb inequality and the random walk representation of the discrete Gaussian free field. 
First of all, we recall here for the reader's convenience part of the statement of the Brascamp--Lieb inequality.
\begin{theorem}[{\citet[Theorem~5.1]{brascamp_lieb}}]\label{thm:BL}
 Let $F(\cdot)$ be a non-negative function on $\R^d$ and let $A$ be a real, positive-definite, $n\times n$ matrix. Assume $\exp(-\la x,\,Ax\ra)F(x)\in L^1$ and define the measure $\mathbb P$ whose density with the respect to the $d$-dimensional Lebesgue measure is
 \[
  \frac{\De\mathbb P(x)}{\De x}:=\frac{\exp(-\la x,\,Ax\ra)F(x)}{\int\exp(-\la x,\,Ax\ra)F(x)\De x}
 \]
 If $F\equiv 1$ we write $\mathbb E_1.$ Let $\phi\in\R^d$, $\alpha\in\R$. Then
 \[
  \mathbb E\left[\left|\la \phi,\,x\ra-\mathbb E[\la\phi,\,x\ra]\right|^\alpha\right]
 \le \mathbb E_1\left[\left|\la\phi,\,x\ra\right|^\alpha\right]\]
 when $F$ is log-concave and $\alpha\ge 1$.
\end{theorem}

Then the following bound is a consequence of the above result, as we are going to show in a moment.
\begin{lemma} For all $x\in\Z$\label{lem:var_bound}
\begin{equation}\label{var_bound}
G_{\Lambda_N}(x,\,x)\le \E_{\Lambda_N}^{GFF}(\vr_x^2).
\end{equation}
Moreover there exists $C>0$ such that for all $x,\,y\in\Z$
\begin{align}\label{eq:grad_var}
\E_{\Lambda_N}[(\vr_{x} - \vr_y)^2] \le C |y-x|.
\end{align} 
\end{lemma}
\begin{proof}
Note that we actually have 
\begin{align*}
H(\vr)|_{\vr\equiv 0 \text{ on } \Lambda_N^c}= \frac{1}{2}\langle\vr,(-\Delta_\Lambda+\sum_{i=2}^K \kappa_i(-\Delta)_\Lambda^{i})\vr\rangle_{\ell^{2}(\Lambda_N)}
\end{align*}
where $\Delta_\Lambda$ and $(-\Delta)_\Lambda^i$ denote the restriction of the operators $\Delta$ and $(-\Delta)^i$ to functions which are zero outside $\Lambda_N$, respectively.
In the Brascamp--Lieb inequality set $$F((\vr_x)_{x\in\Lambda_N}):=\exp\left[-\frac{1}{2}\langle\vr, \sum_{i=2}^K\kappa_i(-\Delta)_\Lambda^{i}\vr\rangle_{\ell^{2}(\Lambda_N)}\right]$$ on $\R^{\Lambda_N}$ with $A:=-1/2\,\Delta_\Lambda$ and $\alpha:=2$.
The first part of the statement is an immediate consequence of Theorem~\ref{thm:BL}. As for the second part, due to the boundary conditions note it suffices to prove~\eqref{eq:grad_var} for $i\in\{1,\,\ldots,\,N-1\}$.
From Theorem~\ref{thm:BL} we have 
\begin{align*}
\E_{\Lambda_N}[(\vr_{x} - \vr_y)^2]\le \E^{DGFF}_{\Lambda_N}[(\vr_{x} - \vr_y)^2] .
\end{align*}
Let $(X_m)_{m=2}^{N-1}$ be a collection of i.i.d.~$\mathcal N(0,2)$ random variables and let $S=(S_i)_{i=1}^{N-1}$ be the simple random walk on $\Z$ with $X_m$'s as increments. We have that the field $(\vr_1,\ldots,\vr_{N-2},\vr_{N-1})$ under $\prob_{\Lambda_N}^{DGFF}$ has the same law of $S$ conditionally on $S_1=S_{N-1}=0$. Now we define the process $(S_1^{'},\ldots,S_{N-1}^{'})$ by $$S_i^{'}:=S_i-\frac{i-1}{N-2}S_{N-1}.$$
As a consequence
$$(S_1,\ldots,S_{N-1}|S_1=S_{N-1}=0)\overset{d}=(S_1^{'},\ldots,S_{N-1}^{'}).$$
Then for $1\le i<j\le N-1$ we have
\begin{align*}
\E[(S_j^{'}-S_i^{'})^2]&=\E\Bigg[\left(\sum_{m=i+1}^{j}X_m -\frac{j-i}{N-2}S_{N-1}\right)^2\Bigg]\\
&= 2(j-i) + 2\frac{(j-i)^2}{N-2} - 2\frac{(j-i)^2}{N-2} 2\\
&=2(j-i)\left[1-\frac{j-i}{N-2}\right].
\end{align*}
This shows the statement.
\end{proof}

 \subsubsection{Proof of Theorem \ref{thm:small_d}} It is easy to see that $(\psi_N(0))_{N\ge 1}$ is tight. Using the properties of Gaussian laws and~\eqref{eq:grad_var}, it can be shown easily that the following holds: there exists $C>0$ such that
	\begin{align}
	\E_{\Lambda_N}\left[ | \psi_N(t)- \psi_N(s)|^2\right]\le C |t-s| \label{eq:moment}
	\end{align}
for all $t,\,s\in\overline D$ uniformly in $N$.
Hence tightness follows. 

To conclude the finite dimensional convergence we first show the convergence of the covariance matrix. Let $G_D$ be the Green's function for the problem 
\begin{align*}
\begin{cases}
-\frac{\De^2}{\De x^2} u(x) = f(x) & x\in D\\
u(x)=0 & x\in\partial D.
\end{cases}
\end{align*}
We note here that $$G_D(x,y)=x\wedge y -xy,\quad x,y\in\overline D$$ which also turns out to be the covariance function of the Brownian bridge, denoted by $(B_t^\circ:0\le t\le 1)$.
For $x,y\in \overline D\cap N^{-1}\Z$ we define
 $$G_{\frac1N}(x,y):=\frac{k^2}{N}G_{\Lambda_N}(Nx,Ny).$$ 
 
 We now interpolate $G_{\frac1N}$ in a piece-wise constant fashion on small squares of $\overline D \times \overline D$ to get a new function $G_{\frac1N}^I$: we define the value of $G_{\frac1N}^I$ in the square $[x, x+1/N)\times [y,y+1/N)$ to be equal to $G_{\frac1N}(x,y)$ for all $x,\,y$ in $\overline{D}\cap N^{-1}\Z$. We show that $G_{\frac1N}^I$ converges uniformly to $G_D$ on $\overline D\times\overline D$. Indeed, let $F_N:=G_{\frac1N}^I - G_D$. From the proof of Proposition \ref{f1} it follows that, for any $f,\,g\in C_c^\infty(D)$,
\begin{align*}
\lim_{N\to\infty}\sum_{x, y\in \frac1N D_N} N^{-2} G^I_{\frac1N}(x,y) f(x) g(y)= \iint_{D\times D} G_D(x,y)f(x)g(y)\De x \De y.
\end{align*}
Again from Riemann sum convergence we have 
\begin{align*}
\lim_{N\to\infty}\sum_{x, y\in \frac1ND_N} N^{-2} G_D(x,y) f(x) g(y)=\iint_{D\times D} G_D(x,y)f(x)g(y)\De x \De y.
\end{align*}
Thus we get
\begin{align}
\lim_{N\to\infty}\sum_{x, y\in \frac1N D_N} N^{-2} F_N(x,y) f(x) g(y)=0\label{eq:unique_zero}.
\end{align}
Note that $G_D$ is bounded and $$\sup_{x,y\in\frac1N D_N}|G_{\Lambda_N}(Nx,Ny)|\overset{\eqref{var_bound}}\le C\sup_{z\in D_N}\E^{GFF}_{\Lambda_N}[\vr_z^2]\le CN.$$
These imply that $$\sup_{x,y\in\overline D} |F_N(x,y)|\le C.$$
Thus $F_N$ has a subsequence converging uniformly to some function $F$ which is bounded by $C$. With abuse of notation we denote this subsequence by $F_N$. We then have 
\begin{align*}
\lim_{N\to\infty}\sum_{x, y\in \frac1N D_N} N^{-2} F_N(x,y) f(x) g(y)= \iint_{D\times D}  F(x,y)f(x)g(y)\De x \De y.
\end{align*}
Uniqueness of the limit gives
$$\iint_{D\times D}  F(x,y)f(x)g(y)\De x \De y = 0$$
 by~\eqref{eq:unique_zero}. From this we obtain that $F(x,y)=0$ for almost every $x$ and almost every $y$. The definition by interpolation of $G^I_{\frac1N}$ ensures that $F$ is pointwise equal to zero. Finally, the fact that the original sequence $F_N$ converges uniformly to zero follows using the subsequence argument. 

We now show the finite dimensional convergence. First let $t\in\overline D$. We write 
$$\psi_N(t)=\psi_{N,1}(t) + \psi_{N,2}(t)$$
where $\psi_{N,1}(t):=kN^{-\frac12}\vr_{\floor{Nt}}$ and $\psi_{N,2}(t):= k N^{-\frac12}(Nt-\floor{Nt})(\vr_{\floor{Nt}+1}-\vr_{\floor{Nt}})$. From \eqref{eq:grad_var} it follows that $\E_{\Lambda_N}[\psi_{N,2}(t)^2]$ goes to zero as $N$ tends to infinity. Therefore to show that $\psi_N(t)\overset{d}\to B_t^\circ$ it is enough to show that $\var[\psi_{N,1}(t)]\to G_D(t,t)$. But we have
\begin{align*}
\var[\psi_{N,1}(t)]=k^2N^{-1}G_{\Lambda_N} \left(\floor{Nt},\,\floor{Nt}\right)=  G^I_{\frac1N}(t,t)\to G_D(t,t)
\end{align*}
since the sequence $F_N$ converges to zero uniformly.
Since the variables under consideration are Gaussian, one can show the finite dimensional convergence using the convergence of the Green's functions. \qed


 
 

\section{Error estimate in the discrete approximation of the Dirichlet problem}\label{sec:Thomee}
This section is devoted to showing that the solution of the continuum Dirichlet problem can be approximated well by the Green's function of the model, and we will give a quantitative meaning to this statement. We shall use the ideas from \cite{thomee}, namely, to employ a truncated operator with which the problems of approximation around the boundary of the discretised domain can be ignored in a nice manner. We recall that the error estimates we mention now were essential to the proof of Theorem~\ref{thm:critical_d}. We begin by introducing some definitions.

In this section we consider $D$ to be any bounded domain in $\mathbb{R}^d$ with boundary $\partial D$ which is $C^2$. We consider the following continuum Dirichlet problem
\begin{equation}\label{eqa:continuum}\begin{cases}
Lu(x) = f(x)& x\in D\\
u(x)=0 & x\in \partial D .
\end{cases}
\end{equation}
where $L$ is the elliptic differential operator $L:=-\Delta_c $.

Let $h>0$ and consider
$$L_h=\sum_{i=1}^K \kappa_i \left( \frac{h^2}{2d}\right)^{i-1} (-\Delta_h)^i,$$
where $\Delta_h$ is defined by
$$\Delta_h f(x):=\frac{1}{h^2}\sum_{i=1}^d\left(f(x+he_i)+f(x-he_i)-2f(x)\right)$$
and $f$ is any function on $h\Z^d$. Let $D_h$ be the set of grid points in $\overline D$ with $D_h= \overline D\cap h\Z^d$. $R_h$ be the largest subset of $D_h$ such that $$R_h\cup\partial_{K,h} R_h\subset D_h,$$
where $\partial_{K,h}R_h=\{D_h\setminus R_h:\,  \mathrm{dist}_h(x, R_h)\le K\}$ with $\mathrm{dist}_h(\cdot\,,\cdot)$ being the graph distance in the lattice $h\Z^d$.
Define $B_h= D_h\setminus R_h$. Let $R_h^\ast$ be the largest subset of $R_h$ such that
$$R_h^\ast\cup \partial_{K,h}R_h^{\ast}\subset R_h$$
and define $B_h^\ast= R_h\setminus R_h^\ast$.  For a grid function $f$ we define 
$$R_hf(x)= \begin{cases}
f(x) &  x\in R_h\\
0 & x\notin R_h.
\end{cases}$$

We now define the finite difference analogue of the Dirichlet's problem \eqref{eqa:continuum}. For given $h$, we look for a function $u_h(\xi)$ defined on $D_h$ such that 
\begin{align}
L_hu_h(\xi)=f(\xi), \quad \xi\in R_h. \label{eq:discrete}
\end{align} 
We consider furthermore the boundary conditions
\begin{align}
u_h(\xi)=0 ,\quad\xi\in B_h. \label{eq:discrete boundary}
\end{align}
One can argue that the finite difference Dirichlet problem \eqref{eq:discrete} and \eqref{eq:discrete boundary} has exactly one solution for arbitrary $f$ \cite[Theorem~5.1]{thomee}.

For grid functions vanishing outside $D_h$ we define the norm $\|\cdot\|_h$ by
$$\lVert f\rVert^2_h:= h^d \sum_{\xi\in D_h} f(\xi)^2.$$
Mind that we are using this norm only in the current Section and thus there is no risk of confusion with the norm defined in Subsection~\ref{subsec:defGFF}.
We now prove the main result of this Section.
\begin{theorem}\label{thm:one}
	Let $u\in\mathcal{C}^3(\overline{D})$ be a solution of the Dirichlet's problem~\eqref{eqa:continuum} and $u_h$ be the solution of the discrete problem ~\eqref{eq:discrete} and ~\eqref{eq:discrete boundary}. If $e_h:=u-u_h$ then for sufficiently small $h$ we have
	$$\|R_he_h\|_{h}^2 \le C \left[M_3^2h^2+ h(M_3^2h^4+M_1^2)\right]$$
	where $M_k=\sum_{\lvert\alpha\rvert\le k}\sup_{x\in D}\lvert D^\alpha u(x)\rvert$.
\end{theorem}

\begin{proof}
	We denote by $C$ all constants which do not depend on $u,\, f$. A standard Taylor's expansion gives for all $x\in R_h$
	$$L_hu(x)=Lu(x)+h^{-2} \mathcal R_3(x)$$
	where 
	\begin{equation}\lvert \mathcal R_3(x)\rvert\leq CM_3h^3.
	 \label{eq:R3}
	\end{equation}
So we obtain for $\xi\in R_h$
	\begin{align*}
	L_he_h(\xi)&=L_hu(\xi)-L_hu_h(\xi)\\
	&=Lu(\xi)+h^{-2}\mathcal R_3(\xi)-L_hu_h(\xi)\\
	&=h^{-2}\mathcal R_3(\xi).
	\end{align*}
The {truncated operator} $L_{h,1}$ is defined as follows:
	$$L_{h,1} f(x) := \begin{cases}
	L_h f(x)  & x\in R_h^\ast\\
	h L_h f(x)  & x\in B_h^\ast\\
	0 & x\notin R_h.
	\end{cases}$$
	For $\xi\in R^*_h$ we have
	\begin{align}\label{eq:Rh}
	L_{h,1}R_he_h(\xi)&= L_hR_he_h(\xi)=L_he_h(\xi)=h^{-2}\mathcal R_3(\xi).
	\end{align}
	As the value of the solution of \eqref{eqa:continuum} is known to be zero on the boundary $\partial D$,  we have for $\eta\in B_h$ $$u(\eta)=u_h(\eta)+ \mathcal R_1(\eta)$$ where $\lvert \mathcal R_1(\eta)\rvert\leq CM_1h$. 
	
	For $\xi\in B_h^\ast$ denote by $$S(\xi)= \{ \eta: \eta\in B_h\setminus (B_h\cap \partial D): \mathrm{dist}_h(\eta,\xi)\le K\}.$$ Therefore, for $\xi\in B_h^\ast$,
	\begin{align}
	L_{h,1} R_he_h(\xi)&= h L_h R_he_h(\xi)\nonumber\\
	&=h \{L_he_h(\xi)- h^{-2} \sum_{\eta\in S(\xi)}C(\eta)e_h(\eta)\}\nonumber\\
	&= h^{-1} \mathcal R_3(\xi)+ h^{-1} \mathcal R^{'}_1(\xi)\label{Rprime}
	\end{align}
	where $C(\eta)$ is a constant depending on $\eta$ and 
	\begin{equation}\label{eq:R4}
	 \lvert \mathcal R^{'}_1(\xi)\rvert\leq CM_1h.
	\end{equation}
	Hence
	\begin{align*}
	\|L_{h,1} R_he_h\|_{h}^2 &= h^d\sum_{x\in R_h} (L_{h,1} R_he_h(x))^2\\
	&= h^d \left[ \sum_{x\in R_h^\ast} (L_{h,1} R_he_h(x))^2+ \sum_{x\in B_h^\ast} (L_{h,1} R_he_h(x))^2 \right]\\
	&\stackrel{\eqref{eq:Rh},\,\eqref{Rprime}}{=} h^d \left[ \sum_{x\in R_h^\ast} (h^{-2} \mathcal R_3(\xi))^2 +\sum_{x\in B_h^\ast} (h^{-1} \mathcal R_3(\xi)+ h^{-1} \mathcal R^{'}_1(\xi) )^2\right]\\
	&\stackrel{\eqref{eq:R3},\,\eqref{eq:R4}}{\le} C h^d \left[ \sum_{x\in R_h^\ast} M_3^2h^2 +\sum_{x\in B_h^\ast} (M_3^2h^4+M_1^2)\right]\\
	&\le C\left[M_3^2h^2+ h(M_3^2h^4+M_1^2)\right]
	\end{align*}
	where the last inequality holds as the number of points in $B_h^\ast$ is $O(h^{-(d-1)})$ which follows from \citet[Lemma 5.4]{penrose_rgg} due to assumption of a $C^2$ boundary.
	Finally using Theorem 4.2 and Lemma 3.1 of \cite{thomee} we obtain
	\begin{align}\label{eq:errorbound}
	\|R_he_h\|_{h}^2 \le C\left[M_3^2h^2+ h(M_3^2h^4+M_1^2)\right]
	\end{align}
	which completes the proof.
\end{proof}
\begin{remark}
Note that in the above proof we used Theorem 4.2 of~\cite{thomee} which requires the domain to satisfy a property called $\mathcal B_1^*$. In the same article it is pointed out that for any domain $\mathcal B_1^*$ holds by definition.
\end{remark}

\begin{theorem}\label{thm:spec_gap}
Let $A_h$ be the matrix $h^2L_h$ and let $\mu_1^{(h)}$ be the smallest eigenvalue of $A_h$. Then
\[
\lambda_1 = \lim_{h\to 0} h^{-2}\mu_1^{(h)},
\]
where $\lambda_1$ is the smallest eigenvalue of $-\Delta_c$.
\end{theorem}

The proof of the above result follows by imitating the proof of Theorem 8.1 of~\cite{thomee} which we skip here.


\end{document}